\documentclass[preprint,11pt,lefttitle]{elsarticle}

\nopreprintlinetrue

\usepackage[T1]{fontenc}
\usepackage{ConformalSurfaceSplines}

\usepackage{libertine}
\usepackage[libertine]{newtxmath}

\usepackage[margin=1in]{geometry}
\usepackage{url}
\usepackage{hyperref}
\hypersetup{
    colorlinks=false,
    pdfborder={0 0 0},
    pdfborderstyle={/S/U/W 0},
}

\usepackage[font=small,labelfont=bf,tableposition=top]{caption}

\title{Conformal Surface Splines}

\begin{document}

\author[1]{Yousuf Soliman\corref{cor1}}
\author[2]{Ulrich Pinkall}
\author[3]{Peter Schr{\"o}der}

\affiliation[1]{organization={Side Effects Software, Inc.},
city={Toronto},
country={Canada}}
\affiliation[2]{organization={Technische Universit\"at Berlin},
city={Berlin},
country={Germany}}
\affiliation[3]{organization={Hausdorff Center for Mathematics, University of Bonn},
city={Bonn},
country={Germany}}

\cortext[cor1]{Corresponding author. Email: \href{mailto:yousufs@sidefx.com}{yousufs@sidefx.com} }

\begin{frontmatter}
\begin{abstract}
      {We introduce a family of boundary conditions and point constraints for conformal immersions that increase the controllability of surfaces defined as minimizers of conformal variational problems. Our free boundary conditions fix the metric on the boundary, up to a global scale, and admit a discretization compatible with discrete conformal equivalence. We also introduce constraints on the conformal scale factor, enforcing rigidity of the geometry in regions of interest, and describe how in the presence of point constraints the conformal class encodes knot points of the spline that can be directly manipulated. To control the tangent planes, we introduce flux constraints balancing the internal material stresses. The collection of these point constraints provide intuitive controls for exploring a subspace of conformal immersions interpolating a fixed set of points in space. We demonstrate the applicability of our framework to geometric modeling, mathematical visualization, and form finding.}
\end{abstract}  
\end{frontmatter}
\section{Introduction}
\label{sec:Introduction}
Conformally constrained Willmore surfaces are a geometric model of bivariate splines that we call \emph{conformal surface splines}. { Recall that univariate elastic splines in \(\RR^3\) are defined as the 
constrained minimizers \(\gamma:[0,L]\to\RR^3\) of the elastic energy 
\[
    \mathcal{E}(\gamma) = \int_{0}^{L}\kappa^2~ds,
\]
where \(\kappa\) is the curvature of the curve and \(ds\) denotes the arc-length element. The most common constraint that is incorporated is an arc-length parameterization constraint; since the elastic energy is diffeomorphism invariant, this effectively only fixes the total length of the curve. Additional constraints on the enclosed volume or total torsion of the curve have also been studied~\cite{Chern:2018:CHF}. Planar elastic curves with a length constraint are a popular mathematical model of physical drafting splines.

Generalizing from the univariate setting, for an immersion \(f : \Sigma\to \RR^3\) of a two-dimensional manifold \(\Sigma\) we take the Willmore energy 
\[ \mathcal{W}(f) = \int_{\Sigma}H^2~dA, \]
where \(H\) is the mean curvature of the immersion, 
to play the role of the elastic energy for curves} and the conformal class to play the role of the length constraint. {The need to relax the isometry constraint in the bivariate setting can be justified by noting that in the smooth category there may not even exist any immersed surface realizing a given metric on \(\Sigma\) (\eg, for compact non-positively curved Riemannian 2-manifolds), whereas smooth conformal immersions into \(\RR^3\) are known to exist for all Riemann surfaces~\cite{Garsia:1961:ICR,Ruedy:1971:EOR}. Just as the metric only fixes the total length of a curve, the conformal class only fixes finitely many geometric parameters of the metric.}

These constraints on the metric, in both the univariate and bivariate settings, ensure that excessive smoothing of the surface geometry is avoided (\secref{ConformalParameters}). In order to be useful for practical applications in surface modeling, one important ingredient is missing from previous work on constrained Willmore surfaces: for a workflow similar to the one used in the context of ordinary splines it is necessary that the location of finitely many points on the surface can be freely prescribed. What is needed is a way to specify the geometry of a surface by interpolating point constraints in space. 

Conformal splines have a number of additional handles that allow the geometry of the surface to be controlled in a delicate and precise way. (1) The metric can be specified by the conformal scale factor at a point, controlling the amount of pressure causing the surface to balloon up predictably. Specifying the metric can also be used to force the splines to preserve geometric features in region of interest. (2) Flux forces that balance the internal material stresses can be prescribed instead of point positions; these forces are subject to a balancing condition that is a consequence of the translational invariance of the discrete energy. {To investigate the influence of these parameters and their utility in the computer graphics applications of surface modeling, fairing and form-finding, we study a discrete formulation of constrained Willmore surfaces and discretize these additional control handles. This enables us to perform a numerical search for constrained Willmore surfaces with special properties for applications in surface modeling---beyond applications in computer graphics, we hope that the discrete theory can also be useful for developing intuition or numerically inspired conjectures about the variational nature of constrained Willmore surfaces.}

\begin{figure}[htpb]
  \centering
  \includegraphics[width = \columnwidth]{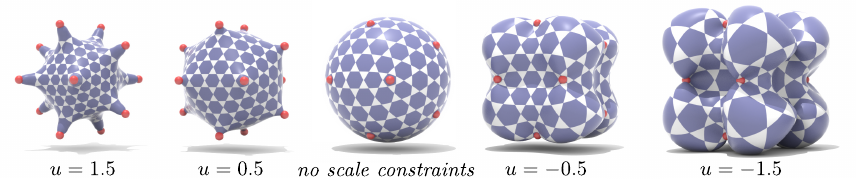}
  \caption{Manipulating the scale factor at the vertices of an icosahedron inscribed in a sphere simultaneously can be used to construct conformal Christmas ornaments.}
  \label{fig:Icosahedron Scale Constraints}
\end{figure}

Boundary conditions also play an important role when working with conformal splines. Existing approaches impose Dirichlet or Neumann boundary conditions, but this is asking too much since it requires the specification of the extrinsic boundary geometry \apriori. On the other hand, some sort of boundary condition is needed for controlling the geometry of the surface well. The boundary condition for conformal immersions used in this paper only fixes the metric on the boundary up to a scale for each connected component. It turns out that this type of boundary condition has many desirable properties for geometric modeling. \figref{StretchShear} illustrates the control afforded by these boundary conditions in ribbon modeling.

\begin{figure}[htpb]
  \centering
  \includegraphics[width = \columnwidth]{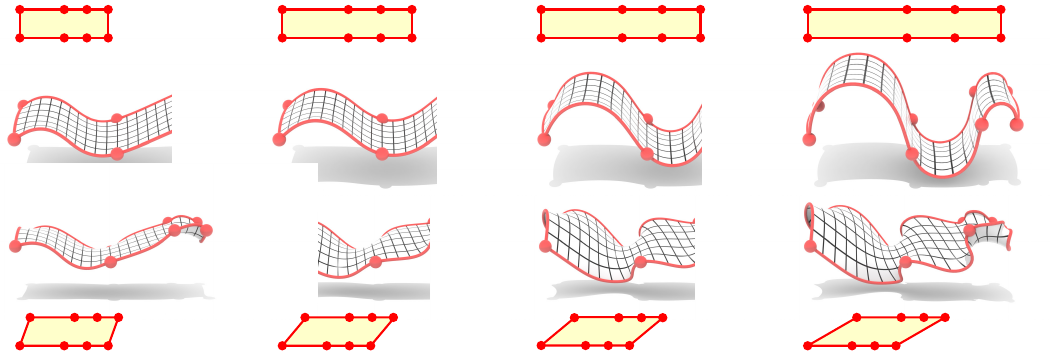}
  \caption{Designing ribbons with conformal surface splines results in a large solution space of interpolating surfaces parameterized by the conformal structure of the immersion punctured at the user-specified point constraints. Fixing the metric on the boundary, up to a scale, provides effective boundary control without explicit specification of the extrinsic geometry. \figloc{Top row: } stretching the conformal type of a rectangular strip results in surfaces that resemble extrusions of planar elastic curves with length constraints. \figloc{Bottom row:} shearing the conformal type produces more complicated deformations.}
  \label{fig:StretchShear}
\end{figure}

\subsection{Related Work}
\label{sec:Related Work}

There is a large body of mathematical research dedicated to the understanding of minimizers of conformal variational problems. Directly related to our work is the study of Willmore and constrained Willmore surfaces. A good summary of the known existence results for minimizers of the Willmore energy under constraints on area and volume can be found in the introduction of \cite{Scharrer:2024:OMW}. The study of constrained Willmore surfaces was initiated in~\cite{Bohle:2008:CWS}, and many transformations and analytic properties of constrained Willmore surfaces are by now well-understood~\cite{Kuwert:2013:MWF,Quintino:2021:CWS}. Although closely related, we are not aware of any mathematical results on the existence of minimizers of the Willmore energy under conformality, point constraints and scale constraints, although the existing results suggest that under fairly general conditions existence of solutions of the variational problem can be expected. 

{In differential geometry the theory of isometric immersions studies the question to what extent the metric of a surface in space can be prescribed. The approaches in~\cite{Chern:2019:FCI,Chern:2018:SfM} aim to compute isometric immersions with a Willmore energy type regularization. In the smooth setting, the Willmore energy is invariant under M{\"o}bius transformations. We preserve this feature by using the M{\"o}bius invariant discretization given in terms of the circumcircle intersection angles~\cite{Bobenko:2005:DWF}.

\paragraph*{Splines}
Due to their ability to smoothly interpolate or approximate function values, splines have been applied in many areas of computer graphics. They are defined by minimizing a smoothness energy subject to interpolation constraints at specified knot points. In many situations, such as for cubic splines, they admit a finite-dimensional solution space and can be represented by piecewise polynomial functions. Univariate elastic splines obtained by minimizing the elastic energy subject to arclength constraints describe a fully nonlinear generalization of cubic splines. Our conformal surface splines can be understood as the natural two-dimensional generalization of these elastic splines: the Willmore energy substitutes the elastic energy and the finite-dimensionality is manifest through conformality. 

Biharmonic smoothing energies (also used for thin plate splines) can be interpreted as the linearization of the Willmore energy and have found a variety of applications in geometry processing~\cite{Jacobson:2011:BBW}. Natural boundary conditions for the biharmonic problem were derived from a variational principle measuring the squared \(L^2\)-norm of the Hessian~\cite{Stein:2018:NBC} and also provide desirable control over the boundary geometry. 

\paragraph*{Conformal Geometry Processing}
Conformal parameterizations of surfaces have found numerous applications in computer graphics. Their most prominent application is for computing texture maps~\cite{Levy:2002:LSC,Mullen:2008:SCP,Springborn:2008:CETM,Sawhney:2017:BFF}, but they also have applications in other areas of geometry processing. One important application includes the registration of signals on surfaces~\cite{Yoshiyasu:2014:ACAP,Baden:2018:MR}. The Yamabe equation, determining the curvature of a conformally rescaled metric, was also used to formulate variational problems for computing optimal cuts and cone singularities on smooth surfaces~\cite{Sharp:2018:VSC,Soliman:2018:OCS,Campen:2019:SPA}. {In the two-dimensional setting, conformal deformations are also used to define and deform domains~\cite{Weber:2010:CCM}---an application not too dissimilar to spline modeling.}  Conformal equivalence has also been extensively studied in discrete differential geometry, where the discretization based on cross ratios preserves much of the structure of conformal maps that is found in the smooth setting~\cite{Springborn:2008:CETM}. The computation of discrete uniformization, in the sense of discrete conformal equivalence, was recently described in \cite{Gillespie:2021:CEPS}. The uniformization domain of the surface defines a Riemann surface and a particular point in Teichm{\"u}ller space, the space of conformal equivalence classes of metrics up to diffeomorphism. Some initial studies on the applications of Teichm{\"u}ller theory to computer graphics were presented in~\cite{Jin:2009:CFNC}. Fenchel-Nielsen coordinates are a standard way to parameterize this space~\cite{Fenchel:2011:DGI}, and they are described relative to a pair of pants decomposition of the surface. An approach to compute pants decompositions based on discrete Morse theory is presented in \cite{Hajij:2016:SSM}.

Although conformal geometry has been at the forefront of geometry processing for the past fifteen years, only a handful of papers consider the interaction between the extrinsic geometry of how a surface is immersed in \(\RR^3\) and the intrinsic conformal geometry of its induced metric.
Discrete spin transformations were introduced in \cite{Crane:2011:STD}, and they parameterize extrinsic conformal deformations in terms of changes in mean curvature half density. They were then used in \cite{Crane:2013:RFC} to compute conformal curvature flows for the purpose of conformal surface fairing. An approach to compute deformations that preserve length cross-ratios was presented in \cite{Vaxman:2015:CMD}, and a M{\"o}bius geometric subdivision scheme was presented in \cite{Vaxman:2018:CMS} that presents an alternative approach to extrinsic conformal geometric modeling. {Most recently, \cite{Corman:2024:CDC} consider a reformulation of the spin transformations that instead works with variations of the shape operator (as opposed to variations of just its trace) and compute curvature driven conformal transformations.} 

\section{Preliminaries}
\label{sec:Preliminaries}
{In this section we will review the notions from smooth and discrete differential geometry needed to define the variational problem that gives rise to constrained Willmore surfaces.}
Throughout \(\Msf = (\Vsf,\Esf,\Fsf)\) will denote a simplicial surface with vertex set \(\Vsf\), unoriented edge set \(\Esf\), and face set \(\Fsf\). We denote the set of oriented edges, or halfedges, by \(\Hsf\). The set of interior vertices, edges and halfedges are denoted \(\Vsf_0\), \(\Esf_0\) and \(\Hsf_0\), respectively. The dual mesh will be denoted by \(\Msf^*\), and we will make the natural identifications \(\Vsf^*\cong\Fsf\), \(\Esf^*\cong\Esf\), and \(\Fsf^*\cong\Vsf\) without further comment.
We will use discrete exterior calculus, denoting the spaces of discrete differential \(k\)-forms, for \(k=0,1,2\), on the primal and dual meshes as \(\Omega^k(\Msf)\) and \(\Omega^k(\Msf^*)\), respectively. Discrete differential forms can be identified with (co-)chains, and the corresponding primal-dual pairing on discrete differential forms will be denoted \(\langle\!\langle \cdot\mid\cdot\rangle\!\rangle\), and the discrete exterior derivative on both primal and dual forms will be denoted \(\dsf\).

\subsection{Discrete Conformal Equivalence}
\label{sec:DiscreteConformalEquivalence}
Consider a smooth surface \(\Sigma\) with a Riemannian metric \(g\). We say that another metric \(\tilde{g}\) is \emph{pointwise conformally equivalent} to \(g\) if there exists a smooth function \(u\colon \Sigma\to\RR\) satisfying \[ \tilde{g} = e^{2u}g. \] Since the inner products are related by a positive scaling, the angles between tangent vectors are preserved. A smooth immersion \(f \colon \Sigma \to\RR^3\) becomes a Riemannian manifold with the pullback metric \(g_{f}  \coloneqq f^*\langle\cdot,\cdot\rangle_{\RR^3} = \langle df,\,df\rangle_{\RR^3}\), and the conformal class of the pullback metric makes \(\Sigma\) a Riemann surface. In the context of geometric variational problems, a conformal constraint amounts to asking that the metric induced by the immersion lies in a prescribed pointwise conformal equivalence class of metrics.
\begin{definition}
  The space of pointwise conformal equivalence classes of Riemannian metrics is denoted 
  \[ 
    \Conf(\Sigma) \coloneqq \{ \text{Riemannian metrics }g\} / \sim
  \]
  where \(g_1\sim g_2\) if the two metrics are pointwise conformally equivalent.
\end{definition}

On the simplicial surface \(\Msf\) a discrete metric is the assignment of edge lengths \(\ell:\Esf\to\RR_{> 0}\) satisfying the triangle inequality on each face---this is equivalent to specifying a Euclidean structure on each triangle. We use the notion of discrete conformal equivalence obtained by rescaling the metric by vertex scale factors~\cite{Luo:2004:CYF,Bobenko:2015:DCM}. 
\begin{definition}
  The space of discrete conformal equivalence classes is 
  \[ 
    \Conf(\Msf) \coloneqq \{ \ell:\Esf\to\RR_{> 0}\} / \sim
  \]
  where \(\ell\sim\tilde{\ell}\) if there exists a function \(\usf:\Vsf\to\RR\) so that for each edge \(\eij\in\Esf\) the metrics are related by \(\tilde{\ell}_{\eij} = \exp\big(\tfrac12(\usf_i + \usf_j)\big)\ell_{\eij}.\)
\end{definition}
Using the logarithmic edge lengths, the relationship between discrete conformal equivalent metrics becomes linear. Define the averaging operator \begin{equation}\Asf:\RR^{|\Vsf|}\to\RR^{|\Esf|},\qquad (\Asf \usf)_{ij} = \tfrac{\usf_i + \usf_j}{2}.\label{eq:averaging_operator}\end{equation} With it, we can equivalently say that two discrete metrics \(\ell,\tilde{\ell}\) on \(\Msf\) are discrete conformal equivalent if there exists some \(\usf\in\Omega^0(\Msf)\) so that \(\log\ell - \log\tilde{\ell} = \Asf\usf\).~\\
\setlength{\columnsep}{1em}
\setlength{\intextsep}{0em}
\begin{wrapfigure}{l}{50pt}
  \includegraphics{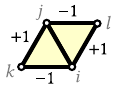}
\end{wrapfigure}
Discrete metrics are conformally equivalent if and only if they have the same length cross ratios on all interior edges~\cite[Proposition 2.3.2]{Bobenko:2015:DCM}. The length cross ratio of an edge \(\eij\in\Esf_0\) incident to oriented faces \(\eijk\) and \(\jil\) is
\begin{equation}
  \operatorname{lcr}_{\eij} = \frac{\ell_{il}\ell_{jk}}{\ell_{lj}\ell_{ki}}.
\end{equation}
{Thus, the logarithmic length cross ratio is a linear function of the logarithmic edge lengths. Define the linear map \(\Csf:\RR^{|\Esf|}\to\RR^{|\Esf_0|}\) by 
\begin{equation}
  \lambda\in\RR^{|\Esf|}\mapsto (\Csf\lambda)_{\eij} = \lambda_{\il} - \lambda_{\lj} + \lambda_{\jk} - \lambda_{ki}
  \label{eq:cross_ratio_operator}
\end{equation}
for each \(\eij\in\Esf_0\) so that \(\log\lcr_{\eij} = (\Csf\log\ell)_{\eij}\).}
\setlength{\intextsep}{1em}

\subsection{Discrete Willmore Energy}
\label{sec:Willmore Energy}
{The Willmore energy of a smooth immersion \(f : \Sigma\to \RR^3\) is invariant under conformal transformations of the ambient space. Up to the integral of the Gaussian curvature, which is a topological constant that we will now include in the definition of \(\mathcal{W}\), the Willmore energy can be written as the integral of a M{\"o}bius invariant integrand 
\[ \mathcal{W}(f) = \int_{\Sigma}(H^2-K)~dA. \] }
~
\setlength{\columnsep}{1em}
\setlength{\intextsep}{0.2em}
\begin{wrapfigure}{r}{80pt}
  \includegraphics{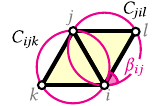}
\end{wrapfigure}
{For an assignment of vertex positions \(\fsf:\Vsf\to\RR^3\) we will use the M{\"o}bius invariant discretization of the Willmore energy introduced by Bobenko~\cite{Bobenko:2005:DWF}.} For each oriented triangle \(\eijk\in\Fsf\) we let \(C_{\eijk}\) denote the circumcircle going through the points \(\fsf_i,\fsf_j,\fsf_k\) in order. For each interior edge \(\eij\in\Esf\), adjacent to faces \(\eijk\) and \(\jil\), define \(\beta_{\eij}\) to be the intersection angle between the circumcircles \(C_{\eijk}\) and \(C_{\jil}\). 
\setlength{\intextsep}{1em}

\begin{definition}
  The Willmore integrand \(W\in\Omega^2(\Msf^*;\RR)\) is the dual 2-form, vanishing on the boundary, defined for each interior vertex \(i\in\Vsf_0\) as
  \[
    W_i \coloneqq \sum_{\eij}\beta_{\eij} - 2\pi.
  \]    
  The discrete Willmore energy is the integral of \(W\) over \(\Msf\):
  \[ \Willmore \coloneqq \sum_{i\in\Vsf}W_i. \]
\end{definition}

\subsection{Discrete Conformal Variational Problems}
\label{sec:ConformalVariationalProblems}
{ A conformal variational problem is defined by a functional \(\mathcal{F}(f)\) on the space of immersions \(f:\Sigma\to\RR^3\) and by a reference metric \(g_0\), defining the prescribed pointwise conformal equivalence class. The question is to find solutions of the optimization problem \[ \min_{f:\Sigma\to\RR^3}\mathcal{F}(f)\quad \text{subject to}\quad [g_f] = [g_0] \text{ in }\Conf(\Sigma).\]
Minimization of the Willmore energy in a fixed conformal class is the prototypical example of a conformal variational problem. The Euler-Lagrange equations for conformal variational problems were derived by Bohle~\etal~\cite{Bohle:2008:CWS}:
\begin{theorem}[\cite{Bohle:2008:CWS}]
  A smooth conformal immersion \(f:\Sigma\to\RR^3\) of a compact Riemann surface \(\Sigma\) is a critical point of \(\mathcal{F}\) under all infinitesimal conformal variations if and only if there exists a holomorphic quadratic differential \(q\in H^0(K^2)\) so that 
  \[ \grad\mathcal{F}(f) = \delta^*(q).\] Here, \(\grad\mathcal{F}(f)\in\Omega^2(\Sigma)\) is the gradient 2-form of \(\mathcal{F}\) defined so that 
  \[
    \mathring{\mathcal{F}}(f) = \int_{\Sigma} \mathring\phi\grad\mathcal{F}(f)
  \]
  for all infinitesimal normal variations \(\mathring{f} = \mathring{\phi}n:\Sigma\to\RR\), and \(\delta^*:H^0(K^2)\to \Omega^2(\Sigma)\) is the adjoint of the map taking infinitesimal normal variations to infinitesimal variations of the almost complex structure~\cite[Eqn. 2.8]{Bohle:2008:CWS}.
\end{theorem} 
}

{A discrete conformal variational problem asks for the minimizer of an functional \(\mathcal{F}(\fsf)\) of the vertex positions \(\fsf:\Vsf\to\RR^3\) subject to the constraint that the induced metric lies in a fixed discrete conformal class described by logarithmic length cross ratios \(\xi_0:\Esf_0\to\RR\)~\cite{Soliman:2021:CWS}:
\begin{equation*}
  \min_{\fsf}\mathcal{F}(\fsf)\quad \text{subject to}\quad \Csf\log\ell(\fsf) = \xi_0.
\end{equation*} 
The first order optimality conditions for such a problem are described by discrete quadratic differentials: assignments \(q:\Esf\to\RR\) satisfying for each interior vertex \(i\in\Vsf_0\) \[\sum_{\eij}q_{\eij} = 0.\]
For each discrete quadratic differential there is an associated extrinsic conformal stress \(\mu\in\Omega^1(\Msf^*;\RR^3)\) defined on each dual edge \(\eij\in\Esf\) by \begin{equation} \mu_{\eij} \coloneqq -q_{\eij} \frac{\dsf\fsf_{\eij}}{|\dsf\fsf_{\eij}|^2}.\label{eq:extrinsicQD} \end{equation} We denote the space of discrete quadratic differentials on \(\Msf\) by \(\QD(\Msf)\). We recall the discrete Euler-Lagrange equations:
\begin{theorem}[\cite{Soliman:2021:CWS}]
  A discrete conformal map \(\fsf:\Vsf\to\RR^3\) is a critical point of \(\mathcal{F}\) under all infinitesimal discrete conformal variations if and only if there exists a discrete quadratic differential \(q\in\QD(\Msf)\) satisfying \(\sum_{\eij}q_{\eij} = 0\) for each boundary vertex \(i\in\Vsf\setminus\Vsf_0\) so that \[ \grad\mathcal{F}(\fsf) = \dsf\mu \] where \(\mu\) is the extrinsic conformal stress defined by \(q\). Here, \(\grad\mathcal{F}(\fsf)\) is the gradient 2-form of \(\mathcal{F}\) defined so that 
  \[
    \mathring{\mathcal{F}}(\fsf) = \sum_{i\in\Vsf}\langle \grad\mathcal{F}(f)_i, \mathring{\fsf}_i\rangle
  \]
  for all infinitesimal variations \(\mathring{\fsf}:\Vsf\to\RR^3\).
  \label{thm:ConformalEL}
\end{theorem}
}
\section{Conformal Constraints}
\label{sec:Constraints}
The inclusion of user-specified constraints plays a prominent role in geometric modeling. In addition to constraints on point positions, area, and enclosed volume there are a number of additional constraints that are only applicable to conformal surface splines.

\subsection{Conformal Knot Points}
\label{sec:ConformalParameters}
{The first handle we will consider are the conformal knot points---these are what we will call the additional conformal moduli needed when specifying a point in the Teichm{\"u}ller space of a punctured surface. To motivate their importance, consider the length constraint for univariate splines.}
More specifically, consider minimizing a geometric bending energy over curves \(\gamma : [0,L]\to\RR^3\) subject to some boundary conditions along with an additional constraint that \(\gamma(s_0) = \gamma_{0}\) for some fixed \(s_0\in[0,L]\) and \(\gamma_0\in\RR^3\). 
Under no additional constraint on the parameterization of the curve, the variational problem remains unchanged if we change the point constraint to \(\gamma(s_1) = \gamma_{0}\) for any other \(s_1\in [0,L]\). 
This is because we can always replace a curve \(\gamma:[0,L]\to \RR^3\) with \(\gamma\circ\Phi:[0,L]\to\RR^3\) where \(\Phi:[0,L]\to [0,L]\) is some orientation preserving diffeomorphism of the interval (\ie, a monotone function) satisfying \(\Phi(s_0) = s_1\). 
More control over an elastic curve spline can be achieved by imposing that the curve is arc length parameterized (or alternatively, that it has some fixed Riemannian metric). This opens up the possibility to edit both the position of the point constraints as well as the length of the curve between successive constraint points. 

Analogously, in the two-dimensional setting the particular choice of constraint points in the parameter domain has no influence on geometric variational problems formulated over the set of surfaces \(f \colon \Sigma\to\RR^3\) constrained to satisfy \(f(p_i) = f_i\) for a collection of finite points \(p_i\in \Sigma\) and \(f_i\in\RR^3\). 
The reason is that a diffeomorphism can be used to maneuver the original constraint points \(p_i\) to an arbitrary configuration~\cite{Michor:1994:nTC}---that is, for every \(n\in\mathbb{N}\) the action of the diffeomorphism group on points is \(n\)-transitive. To control this diffeomorphism degree of freedom that allows any collection of points in the parameter domain to be exchanged with any other, one might attempt to fix the metric exactly as in the univariate setting. This approach is physically motivated, in elasticity thin plate bending can be modeled by bending energy minimizers subject to isometric constraints~\cite{Friesecke:2002:TGR,Bartels:2013:FEM}. Although physically motivated, isometric bending models require the specification of infinitely many parameters (the metric) making it a challenging formulation for two-dimensional splines. The relaxation to a conformal constraint on the parameterization solves this problem: it lies in-between the flexibility admitted by no constraint on the parameterization and the rigidity imposed by isometric deformations, and it is specified by a finite number of parameters. When one considers the conformal class in the presence of point constraints the additional conformal parameters can be interpreted as knot values for the resulting spline.

\begin{figure}[ht]
  \centering
  \includegraphics[width=\columnwidth]{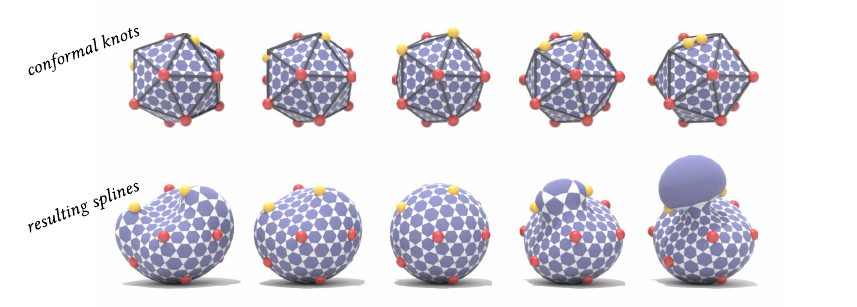}
  \caption{The additional conformal parameters that arise when specifying a surface with point constraints correspond to their positions in the parameter domain. Pulling these conformal knots together in \(M\) results in a pulling more material between the points extrinsically. \label{fig:ConformalKnots}}
\end{figure}

A precise description of these conformal parameters is given by Riemann's moduli space of conformal classes. Recall that we can pull back a Riemannian metric \(g\) by a diffeomorphism \(\Phi\in\Diff(\Sigma)\) to obtain a new metric \(\Phi^*g\).
\begin{definition}
The \textbf{moduli space} of conformal classes on \(\Sigma\) is the quotient of \(\Conf(M)\) by diffeomorphism: \[\Moduli(\Sigma) \coloneqq \Conf(\Sigma)/\Diff(\Sigma).\]
\end{definition}
In the discrete setting, the action by diffeomorphism can be thought of as an intrinsic remeshing of the surface. Since geometric functionals are invariant under diffeomorphism a constraint on \(\Conf(\Sigma)\) can be replaced by a constraint on \(\Moduli(\Sigma)\) without changing the variational problem. 
The reason we are interested in this quotient space is that because, unlike \(\Conf(\Sigma)\) which is infinite-dimensional, the moduli space \(\Moduli(\Sigma)\) of a genus \(g\) surface is a \(6g-6\) dimensional space~\cite{Thurston:1997:GT3}. To specify a point and navigate in this space, one can use Fenchel-Nielsen coordinates, but we do not pursue how to navigate this space by modification of the length cross ratios further. Instead, we focus on navigating the moduli space in the presence of point constraints which necessitates working with the moduli space of a punctured surface. 

When the positions of points are also prescribed, \(f(p_i) = f_i\in\RR^3\) for \(i=1,\dots,n\), not all diffeomorphisms will preserve this constraint. Therefore, we cannot replace the pointwise conformal constraint with a constraint on \(\Moduli(\Sigma)\). Instead, we can replace it by a constraint on the moduli space of the surface punctured at the constraint points \(\Sigma^* \coloneqq \Sigma\setminus\{p_1,\dots,p_n\}\). The space \(\Moduli(\Sigma^*)\) is a fiber bundle over \(\Moduli(\Sigma)\) with fiber the space of configurations of \(n\) labeled points on \(\Sigma\)~\cite{Farb:2011:PMC}. Consequentially, specifying a conformal structure on the punctured surface is the same thing as prescribing the conformal structure of \(\Sigma\) along with the positions of \(n\) points in \(\Sigma\)---we call these points the \emph{conformal knot points} of the spline. In \figref{ConformalKnots} we visualize the deformation obtained by pulling the conformal knot points together and apart, while keeping their positions in space fixed.

\subsection{Total Torsion}
{The conformal type also influences the extrinsic geometry. Below we explain its influence for thin tubes around space curves, where it ensures that they remain approximately a constant thickness when restricted to the conformal class. 
This property can be quite fragile and will only be preserved if the conformal class and thickness of the curve do not change along the flow---\figref{Bubbling} shows that bubbling phenomena readily materialize when the conformal constraint is the only active constraint. 
To remedy this instability, we also constrain both the area and the enclosed volume of the surface---Gruber and Aulisa
previously showed that constraining just these two scalar quantities is enough to ensure the stability of thin tubes without any conformality constraint~\cite{Gruber:2020:CpW}. 
So what does the conformal class add in this setting? Numerical experiments suggest that it additionally fixes the total torsion of the curve---below, we will justify these results with approximations of the conformal modulus for thin tori. }

Consider an immersion \(\gamma: S^1\to\RR^3\), we get a 1-form \(ds = |d\gamma|\) on \(S^1\). Given in addition a thickness function \(a:S^1\to\RR\), in the limit of small and slowly varying \(a\), we can compute various global invariants (area, volume, Willmore energy) of the tube around \(\gamma\) with thickness \(a\):
\begin{align*}
  \mathcal{A} & \approx 2\pi\int_{S^1} a~ds, &
  \mathcal{V} & \approx \pi \int_{S^1} a^2~ds, &
  \mathcal{W} & \approx \frac{\pi}{2}\int_{S^1}\left(\tfrac{1}{a} + \frac{a \kappa^2}{2}\right) ds.
\end{align*}
The above approximate equalities already explain the mentioned phenomenon nicely: if we fix a thickness function \(a\), then (in the limit of small \(a\) we can forget about the term that involves \(\kappa\)) a critical point has to have constant thickness \(a\). 

\begin{figure}[!h]
  \includegraphics[width=\columnwidth]{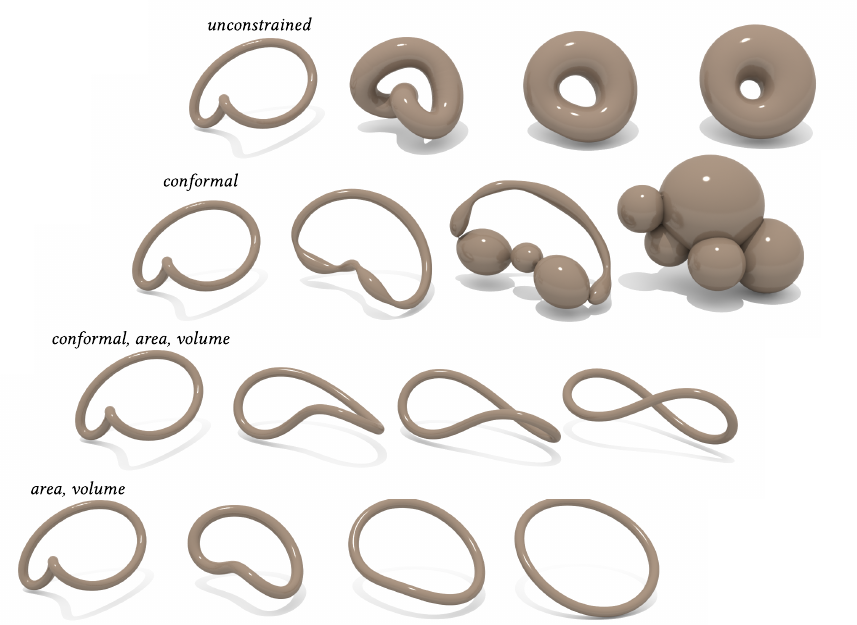}  
  \centering
  \caption{From top to bottom. \\\figloc{Row 1:} unconstrained Willmore minimization. \\\figloc{Row 2:} Willmore minimization under a soft conformal constraint. Bubbling occurs even though the surface remains approximately conformal throughout the flow. \\\figloc{Row 3:} Willmore minimization under area, volume, and conformality constraints. All three of these constraints ensure both the stability of the thin tube around a centerline and the total torsion of the curve. \\\figloc{Row 4:} Willmore minimization under area and volume constraints. The stability of the thin tube persists, but the total torsion of the curve is not conserved.}
  \label{fig:Bubbling}
\end{figure}

As a conformal torus, our tube with varying thickness will be conformally equivalent to \(\CC/\Gamma\) where the lattice \(\Gamma\) is generated by \(2\pi\in\RR\subset\CC\) and another generator \(\tau\in\CC\). 
Now we will compute an approximation of \(\tau\) from the extrinsic geometry of the surface. Choosing a parallel frame (with monodromy angle \(\Theta\) equal to the total torsion) along \(\gamma\), we obtain the differntial \(d\varphi\) of an angle coordinate \(\varphi\) on the tube of varying thickness. Under our assumptions, the 1-form \(d\varphi\) will be approximately harmonic with period \(2\pi\) in the \(\varphi\)-direction and with period \(\Theta\) along the \(s\)-direction. This already gives us the real part of \(\tau\): \[\operatorname{Re}\tau = \Theta.\] The integral of \(*d\varphi\) vanishes on the short loops in the \(\varphi\)-direction. The integral of \(*d\varphi\) around a loop in the \(s\)-direction is \[\operatorname{Im}\tau = \int_{S^1}\frac{1}{a}~ds.\]
This is interesting since it implies that for the purpose of stabilizing constant thickness, the conformal type (in the form \(\operatorname{Im}\tau\)) has a similar effect as the Willmore functional.

When we do experiments where the conformal structure is fixed, we are effectively using not parameterization by arclength \(s\) but parameterization by \emph{conformal modulus} \(y\) which (up to an additive constant) is given by \[ dy = \frac{1}{a}~ds. \]
In other words, \[ds = a~dy\] and the area and the volume can be written as
\begin{align*}
  \mathcal{A} & \approx 2\pi\int_{S^1} a^2~dy, &
  \mathcal{V} & \approx \pi \int_{S^1} a^3~dy, &
  \mathcal{W} & \approx \frac{\pi}{2}\int_{S^1}\left(1 + \frac{a^2 \kappa^2}{2}\right) dy.
\end{align*}
Though admittedly crude, these approximations provide useful intuition about the extrinsic geometric features that are fixed by the conformal class.

\subsection{Flux Constraints}
\label{sec:FluxConstraints}
{
Since, up to topological constants, the Willmore energy is equal to the Dirichlet energy of the Gauss map, it is natural to ask whether it is possible to simultaneously prescribe the position and tangent plane of a point in a variational problem. Just as it is not possible to prescribe the positions of isolated points for minimal surfaces, this problem is not well-posed as stated. One way to practically resolve this is by replacing the point constraint with a Dirichlet boundary condition after cutting out a disk around the desired point (\figref{normal_prescription}). The radius of this disk is a regularization parameter that decreases along with decreased influence of the prescribed normal. 

\begin{figure}[htpb]
  \centering
  \includegraphics[width=\columnwidth]{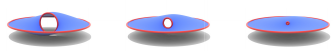}
  \caption{Prescribing the normal vector at an isolated point in Willmore minimization is not well-posed, illustrated by the sequence of Willmore surfaces converging to a flat disk with a shrinking boundary component encoding the desired normal direction.\label{fig:normal_prescription}}
\end{figure}

Although we cannot simultaneously prescribe the tangent plane and position of a point, we can prescribe the tangent plane alone during Willmore minimization by taking advantage of certain conservation laws of the Willmore energy. In the discrete setting, we
consider the prescription of the Willmore flux at isolated points \(i\), meaning that  \[ \grad\Willmore_i = \nu_i,\] \(\nu_i\in\RR^3\). Since the gradient of the Willmore energy defines a vertex normal vector, prescribing \(\nu_i\) fixes the tangent plane of the point. The flux constraint can alternatively be interpreted as specifying the Lagrange multiplier associated with a point constraint instead of the position of the point. 
}

\paragraph*{Balancing Condition}
{ In \secref{WillmoreConservationLaws} we will prove that the gradient of the discrete Willmore energy} can be written in a discrete divergence form
\[ \grad\Willmore = \dsf\tau \]
for some \(\tau\in\Omega^1(\Msf^*;\RR^3)\) (\prpref{WillmoreFlux}) that we call the Willmore flux form. It expresses the conservation law coming from the translational invariance of the energy, which implies that valid flux configurations need to satisfy a balancing condition. 

At a critical point, \(\dsf\tau = 0\) which implies that \(\tau\) defines a cohomology class \([\tau]\in H^1(\mesh^*;\RR^3)\) which we call the \emph{Willmore flux} of the surface. If point or flux constraints are in play then \(\dsf\tau_i = \nu_i\) at the constraint vertices \(i\). Since \(\tau\) is still closed away from these points it defines a cohomology class on the surface punctured at these points (\cf \secref{ConformalParameters}). Therefore, the homology of the surface will imply a necessary balancing conditions on \(\nu_i\) for them to be the Lagrange multipliers of a critical point. 

For simplicity, suppose that \(\Msf\) is topologically a sphere. The homology of a sphere punctured at \(n\) points is \((n-1)\)-dimensional, and a presentation of this group is given by the loops around each point constraint with the only relation that the sum of all of these cycles is equal to zero. Since \(\nu_i\) is the integral of \(\tau\) around the boundary of the dual face containing \(i\), this relation implies that 
\begin{equation}\sum_{i}\nu_i = 0.\label{eq:balancing_equation}\end{equation}
In the higher genus case, we use that when \(n > 1\), a genus \(g\) surface with \(n\) points removed is homotopy equivalent to a wedge of \(2g + n - 1\) circles, and a simple presentation of the homology group is given by the \(2g\) homology generators of \(\Msf\) and \(n\) loops around each puncture with the relation that the sum of all of these cycles is equal to zero. This relation then implies that 
\begin{equation}\sum_{i}\xi_i + \sum_{j=1}^{2g}\int_{\gamma_j}\tau = 0,\label{eq:balancing_equation_general}\end{equation} 
where \(\{\gamma_j\}_{j=1}^{2g}\) are generators of \(H^1(\Msf)\) realized as oriented 1-chains in the dual mesh. 

\begin{figure}[htpb]
  \centering
  \includegraphics[width=\columnwidth]{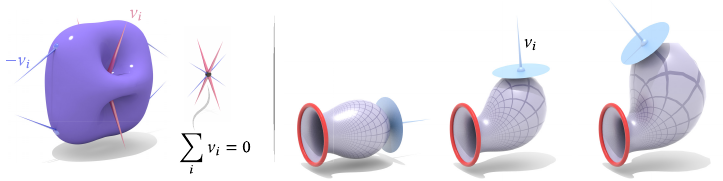}
  \caption{\figloc{Left:} The fluxes \(\nu_i\) at point constraints are visualized with blue vectors indicating a negative scale and red ones indicating a positive scale. Homology then implies that the sum of the fluxes of a constrained Willmore sphere vanish. \figloc{Right: }Specifying the Willmore flux provides a natural deformations between constrained Willmore surfaces obtained by manipulating an external force pulling on the surface.\label{fig:flux}\label{fig:FluxBalancing}}
\end{figure}

\paragraph*{Constrained Willmore Surfaces}
For a constrained Willmore surface \(\tau\) is no longer closed, but since the Lagrange multipliers coming from the conformal constraint only change the Euler-Lagrange equations by the exterior derivative of an extrinsic conformal stress \(\mu\) we instead have that \(\tau - \mu\) is closed. Just as above, we then have a cohomology class containing this closed form that we will call \([\tau_{\mu}]\in H^1(\mesh^*;\RR^3)\). The balancing conditions are equally applicable for a constrained Willmore surface with point constraints if one takes into account the addition of \(\dsf\mu\) to the Euler-Lagrange equations. {Observe, however, that the incorporation of the extrinsic conformal stress may result in \(\nu_i\) no longer being a normal vector of the resulting surface (\figref{FluxBalancing}).}

\subsection{Scale Constraints}
\label{sec:Scale Constraints}
For conformal immersions the metric in the tangent space of a point, relative to another metric in the conformal class, can be specified by a scalar quantity: the logarithmic conformal scale factor. Prescribing the scale factor at a point provides a useful modeling tool where the area distortion around a point can be controlled. This allows for the specification of the size of a bump (\figref{PlanarScaleConstraint}) and metric rigidity (\figref{FixedMetric}).

\begin{figure}[htpb]
  \centering
  \includegraphics[width=\columnwidth]{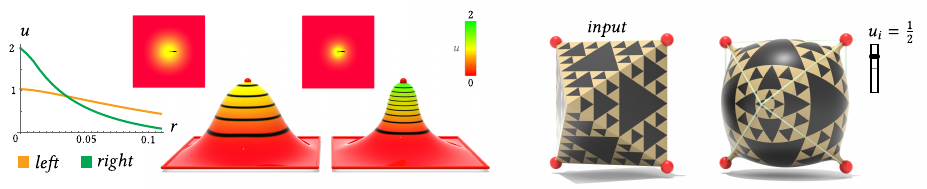}
  \caption{\figloc{Left:} Manipulating the conformal scale factor at a point can be used to control the size of an interpolating bump. Plotting the falloff of the conformal scale factor (relative to the flat square) near the constraint point provides a description of the bump. \figloc{Right:} A discrete constrained Willmore surface with scale constraints conformally equivalent to the regular octahedron. \label{fig:PlanarScaleConstraint}}
\end{figure}

\paragraph*{Discretization}
To discretize the conformal scale factor constraint we can use an averaging procedure. Any triangle \(\eijk\) with two discrete metrics \(\ell,\tilde{\ell}\) are automatically conformally equivalent \(\ell = e^{u}\tilde\ell\) with vertex scale factors \[ \usf_i^{\jk} \coloneqq (\log{\ell}_{\eij} - \log{\ell}_{\jk} + \log{\ell}_{\ki}) - (\log\tilde{\ell}_{\eij} - \log\tilde{\ell}_{\jk} + \log\tilde{\ell}_{\ki}). \]
\setlength{\columnsep}{1em}
\setlength{\intextsep}{0em}
\begin{wrapfigure}{r}{45pt}
   \includegraphics{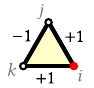}
\end{wrapfigure}
Therefore, the conformal scale factor of a discrete metric \(\ell\) relative to a fixed metric \(\tilde{\ell}\) in the conformal class can be computed by averaging these per-face vertex scale factors 
\[ \usf_i \coloneqq \tfrac{1}{n_i}\sum_{\eijk}\usf_i^{\jk}, \] where \(n_i\) is the valence of \(i\). 
The right-hand side is obviously defined for all metrics, conformal or not, and so it can be incorporated as an additional constraint to the variational problem.
\setlength{\intextsep}{1em}

For a given vertex \(i\in\Vsf\) we can express the scale factor as a linear function of the logarithmic edge lengths, \(\usf_i = \Usf_i(\log\ell - \log\tilde\ell)\), where \(\Usf_i\in\RR^{1\times|\Esf|}\) is given by summing local \(1\times 3\) contributions associated to each face containing the vertex: \[\Usf_i \coloneqq \begin{blockarray}{cccc}
  \textcolor{commentblue}{\eij} & \textcolor{commentblue}{\jk} & \textcolor{commentblue}{\ki} \\
  \begin{block}{(ccc)c}
    1 & -1 & 1 & \\
  \end{block}
  \end{blockarray}.\]  A collection of \(m\) point constraints can be specified in terms of the matrix \(\Usf\in\RR^{m\times|\Esf|}\) where the row corresponding to the scale constraint at \(i\in\Vsf\) is equal to \(\Usf_i\).
\paragraph*{Scale Adjustment}
\setlength{\columnsep}{1.5em}
\setlength{\intextsep}{0em}
\begin{wrapfigure}{r}{50pt}
   \includegraphics{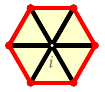}
\end{wrapfigure}
The Willmore minimizers when the conformal scale factor at isolated vertices are prescribed are smooth throughout the surface, but exhibit very sharp cones near the scale constraint points (\figref{ScaleBubbling}).
This undesirable behavior can be easily resolved by instead fixing the conformal scale factor at all the adjacent vertices to the desired vertex, leaving the central vertex out (red vertices, inset). 
This fixes the metric on a small circle around the constraint point, and prevents the sharp cone artifact from developing.
\setlength{\intextsep}{1em}

\begin{figure}[htpb]
  \centering
  \includegraphics{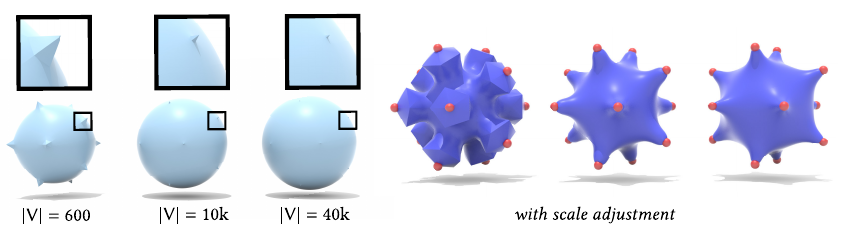}
  \caption{We fix the positions of points on \(S^2\) at the vertices of an icosahedron and specify a logarithmic discrete conformal scale factor of \(u = 1\), relative to the round metric on \(S^2\), at all of these points. \figloc{Left:} undesirable sharp features that do not disappear under refinement arise if one only prescribes the vertex scale factor at isolated vertices. \figloc{Right:} prescribing the metric on the link instead produces smooth surfaces. \label{fig:ScaleBubbling}}
\end{figure}

\section{Conservation Laws}
\label{sec:WillmoreConservationLaws}

In this section, we derive the conservation laws for the discrete Willmore energy that come from its invariance under M{\"o}bius transformations. We start with the conservation law induced by the translation invariance, which states that the gradient 2-form is exact. We will identify \(\RR^3\) with the imaginary quaternions; for a vector \(x\in\RR^3\), we denote its quaternionic inverse by \(x^{-1}\coloneqq -x/|x|^2\) and its normalization by \(\hat{x} = \tfrac{x}{|x|}\).

{
\paragraph{Quaternions, Circles, and Spheres}
\setlength{\columnsep}{1.5em}
\setlength{\intextsep}{0em}
\begin{wrapfigure}{r}{80pt}
   \includegraphics{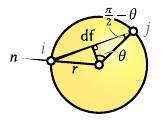}
\end{wrapfigure}
\setlength{\intextsep}{1em}
We will need a couple of standard facts about the quaternionic product of edge vectors around a triangle and an edge---proved, for example, in~\cite{Muller:2021:DCT}. 
We start though by recalling the elementary geometric fact that the curvature \(\kappa\) of a sphere \(S\) going through two points \(\fsf_i,\fsf_j\in\RR^3\) can be computed as 
\[
  \kappa = 2\langle n, \dsf\fsf^{-1}\rangle,
\]
where \(n\) is the normal of \(S\) at \(\fsf_i\) and \(\dsf\fsf = \fsf_j - \fsf_i\).\footnote{{With the notation in the inset diagram, \(|\dsf\fsf| = 2r\sin\theta = 2r\cos(\tfrac\pi2 - \theta) = -2r\langle n,\tfrac{\dsf\fsf}{|\dsf\fsf|}\rangle\). Hence, \(\kappa = \tfrac1r = 2\langle n,-\tfrac{\dsf\fsf}{|\dsf\fsf|^2}\rangle = 2\langle n, \dsf\fsf^{-1}\rangle\).}}

\begin{lemma}[\cite{Muller:2021:DCT} Lemma 9, 10]
  Let \(\fsf_i,\fsf_l,\fsf_j,\fsf_k\in\RR^3\) be a four points not lying on a common circle. Consider the circumcircle of \(\fsf_i, \fsf_j, \fsf_k\) oriented according to this ordering, and let \(t_{\eijk}^i\) be the tangent of this circle at \(\fsf_i\) (similarly, let \(t_{\jil}^i\) be the tangent of the oriented circle through \(\fsf_j,\fsf_i,\fsf_l\)). Let \(\beta_{\eij}\in[0,\pi]\) be the circumcircle intersection angle (the angle between \(t_{\eijk}^i\) and \(t_{\jil}^i\)). Consider the circumsphere (or plane) of the four points oriented so that \(t_{ijk}^i, t_{\jil}^i\) forms a positively oriented basis of the tangent space of the sphere at \(\fsf_i\), and let \(n_{\eij}^i\) be the outward pointing unit normal vector of the sphere at \(\fsf_i\).  Then \[\widehat{\dsf\fsf}_{\eij}\widehat{\dsf\fsf}_{\jk}\widehat{\dsf\fsf}_{\ki} = t_{\eijk}^i, \qquad \widehat{\dsf\fsf}_{\il}\widehat{\dsf\fsf}_{\lj}\widehat{\dsf\fsf}_{jk}\widehat{\dsf\fsf}_{ki} = -\cos\beta_{\eij} + \sin\beta_{\eij}\, n_{\eij}^i.\] 
  \label{lem:HCrossRatio}
\end{lemma}
}
\begin{proposition}
  For \(\fsf:\Vsf\to\RR^3\), define the Willmore flux form \(\tau\in\Omega^1(\Msf^*;\RR^3)\) by
  \begin{equation}\tau_{\eij} \coloneqq 2(n_{\il}^i - n_{\lj}^{j} + n_{\jk}^{j} - n_{\ki}^{i})\times \dsf\fsf_{\eij}^{-1}.\label{eq:WillmoreFlux}\end{equation} Here \(n_{\eij}^{p}\) denotes the normal of the circumsphere of \(\fsf_i,\fsf_j,\fsf_k,\fsf_l\) evaluated at the point \(\fsf_p\).
  Then
  \[ \grad\Willmore = \dsf \tau. \]
  \label{prp:WillmoreFlux}
\end{proposition}
\begin{proof}
  Consider a variation \(\mathring{\fsf}:\Vsf\to\RR^3\). Then 
  \[ \mathring{\mathcal{W}} = \sum_{i\in\Vsf}\sum_{j}\mathring{\beta}_{ij} = 2\sum_{ij\in\Esf}\mathring{\beta}_{ij} \] since \(\beta_{ij} = \beta_{ji}\).
  To compute \(\mathring{\beta}_{ij}\) differentiate the real part of the normalized cross ratio.  {By \lemref{HCrossRatio} the circumcircle intersection angle can be computed as}
  {\InsertBoxR{1}{\begin{minipage}{80pt}\centering
    \includegraphics{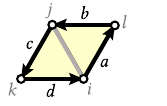}
  \end{minipage}}[1]}
  \[\beta_{ij} = \beta(\dsf\fsf_{il},\, \dsf\fsf_{lj},\, \dsf\fsf_{jk},\, \dsf\fsf_{ki})\] 
  where \(\beta : \RR^3\times \RR^3\times\RR^3\times\RR^3\to {[0,\pi]}\) is defined by \[\cos\big(\beta(a,b,c,d)\big) = -\operatorname{Re}(\,\hat{a}\hat{b}\hat{c}\hat{d}\,).\]
  The variations of \(\beta(a,b,c,d)\) in terms of \(\mathring{a}\) and \(\mathring{d}\) can be computed in a symmetric way (suppressing the dependence on the parameters below) {when \(\mathring{b} = \mathring{c} = 0\):} \[ -\sin(\beta)\, \mathring{\beta} = -\operatorname{Re}\big(\,\mathring{\hat{a}}\hat{a}^{-1}(\hat{a}\hat{b}\hat{c}\hat{d}) + (\hat{a}\hat{b}\hat{c}\hat{d})\hat{d}^{-1}\mathring{\hat{d}}~\big). \] Since \(\hat{a}\) is a unit vector its variations lie in its orthogonal complement, which implies that \( \mathring{\hat{a}}\hat{a}^{-1} = \mathring{\hat{a}}\times \hat{a}^{-1} = \mathring{a}\times a^{-1}\), and a similar argument shows that \(\hat{d}^{-1}\mathring{\hat{d}} = d^{-1}\times \mathring{d}\). Therefore, \begin{align*}
    -\sin(\beta)\, \mathring{\beta} & = -\operatorname{Re}\big(\, (\mathring{a}\times a^{-1})(\hat{a}\hat{b}\hat{c}\hat{d}) + (\hat{a}\hat{b}\hat{c}\hat{d})(d^{-1}\times\mathring{d})~\big) = \langle\operatorname{Im}(\hat{a}\hat{b}\hat{c}\hat{d}),~\mathring{a}\times a^{-1} + d^{-1}\times\mathring{d}\rangle \\ & = \langle a^{-1} \times \operatorname{Im}(\hat{a}\hat{b}\hat{c}\hat{d}),~\mathring{a}\rangle  - \langle d^{-1} \times \operatorname{Im}(\hat{a}\hat{b}\hat{c}\hat{d}),~\mathring{d}\rangle\end{align*} where in the last step we used \(\langle a\times b,c\rangle = \langle a,b\times c\rangle\).
    
    The fact that \(\beta_{ij} = \beta_{ji}\) is reflected in the symmetry of the cross ratio \(\beta(a,b,c,d) = \beta(c,d,a,b)\). This symmetry implies that the variations of \(\beta(a,b,c,d)\) in terms of \(\mathring{b}\) and \(\mathring{c}\) can also be computed in the same symmetric way {when \(\mathring{a} = \mathring{d} = 0\):} \[ -\sin(\beta)\, \mathring{\beta} = \langle c^{-1} \times \operatorname{Im}(\hat{c}\hat{d}\hat{a}\hat{b}),~\mathring{c}\rangle  - \langle b^{-1} \times \operatorname{Im}(\hat{c}\hat{d}\hat{a}\hat{b}),~\mathring{b}\rangle. \] {By \lemref{HCrossRatio},} this then implies that the variation of \(\beta_{ij}\) is equal to 
    \begin{align*}
    -\mathring{\beta}_{ij} & =  \langle \dsf\fsf_{il}^{-1} \times n_{ij}^{i},~\dsf\mathring{\fsf}_{il}\rangle  - {\langle \dsf\fsf_{lj}^{-1} \times n_{ij}^{j},~\dsf\mathring{\fsf}_{lj}\rangle} +  \langle \dsf\fsf_{jk}^{-1} \times n_{ij}^{j},~\dsf\mathring{\fsf}_{jk}\rangle - {\langle \dsf\fsf_{ki}^{-1} \times n_{ij}^{i},~\dsf\mathring{\fsf}_{ki}\rangle}.
    \end{align*}
    Substituting this into the original expression of \(\mathring{\mathcal{W}}\) shows that 
\begin{align*}\mathring{\mathcal{W}} & = 2\sum_{ij\in E}\Big( \langle n_{ij}^{i}\times \dsf\fsf_{il}^{-1},~\dsf\mathring{\fsf}_{il}\rangle  - \langle n_{ij}^{i}\times \dsf\fsf_{ki}^{-1},~\dsf\mathring{\fsf}_{ki}\rangle +  \langle n_{ij}^{j}\times \dsf\fsf_{jk}^{-1},~\dsf\mathring{\fsf}_{jk}\rangle  - \langle n_{ij}^{j}\times \dsf\fsf_{lj}^{-1},~\dsf\mathring{\fsf}_{lj}\rangle\Big) \\ & = \sum_{ij\in E}\langle {-}2(n_{il}^i - n_{lj}^{j} + n_{jk}^j - n_{ki}^i)\times \dsf\fsf_{ij}^{-1},\,\dsf\mathring{\fsf}_{ij}\rangle = {-}\sum_{\eij\in\Esf}\langle\dsf\mathring{\fsf}_{\eij}, \tau_{\eij}\rangle = {\sum_{i\in\Vsf}\langle \mathring{\fsf}_i, \sum_{\eij}\tau_{\eij}\rangle} \\ & {= \sum_{i\in\Vsf}\langle\mathring{\fsf}_i, \dsf\tau_i\rangle = \langle\!\langle \dsf\tau \mid \mathring{\fsf}\rangle\!\rangle,} \end{align*} {showing that \(\grad\Willmore = \dsf\tau\), as desired.}
\end{proof}

\begin{remark}
In \cite{Bobenko:2005:DWF} the authors showed that there is a gradient singularity at \(\beta_{\eij} = 0\) and introduce a simple heuristic to remedy this. Whenever \(|\sin\beta_{\eij}|<\varepsilon\), set the gradient to be equal to zero. In the notation above, this corresponds to setting the undefined circumsphere normals to be zero. We take \(\varepsilon = 10^{-6}\) in double precision. 
\end{remark}
Noether's theorem tells us that the pairing between the gradient 2-form and the surface deformation induced by an infinitesimal M{\"o}bius transformation is exact. For an infinitesimal translation in a direction \(v\in\RR^3\) the conservation law states that \(\langle\grad\Willmore,v\rangle\) is exact, and in \prpref{WillmoreFlux} we explicitly found the potential \(\langle \tau, v\rangle\). 
The other infinitesimal M{\"o}bius transformations are described as follows. The deformation induced by an infintiesimal rotation is given by \(\mathring{f} = w\times f\) for some \(w\in\RR^3\), and the deformation induced by an infinitesimal scaling is given by \(\mathring{f} = \lambda f\) for some \(\lambda > 0\). The remaining M{\"o}bius degree of freedom is given by the composition of a sphere inversion, a translation, and another sphere inversion. The deformation induced by such an infinitesimal M{\"o}bius transformation is \(\mathring{f} = f\times(f\times u) + f\langle f, u\rangle\) for some \(u\in\RR^3\). An explicit expression for the conservation laws associated to these deformations is given in the following proposition.
\begin{theorem}
  Define \(\sigma\in\Omega^1(\Msf^*)\) and \(\rho,\zeta\in\Omega^1(\Msf^*;\RR^3)\) by 
  \begin{align*}
    \sigma_{\eij} & = \langle \fsf_{\eij}, \tau_{\eij}\rangle \\ 
    \rho_{\eij} & = \fsf_{\eij}\times\tau_{\eij} + \tfrac{1}{2}H_{\eij}\dsf \fsf_{\eij} \\
    \zeta_{\eij} & = \tfrac{1}{2}(\fsf_i\times (\fsf_j\times\tau_{\eij}) + \fsf_j\times(\fsf_i\times\tau_{\eij})) + \fsf_{\eij}\langle \fsf_{\eij},\tau_{\eij}\rangle + \fsf_{\eij}\times(H_{\eij}\dsf \fsf_{\eij})
  \end{align*}
  for every \(\eij\in\Esf\). Here, we denote by \(\fsf\in\Omega^0_{\Esf}(\Msf;\RR^3)\) (see \ref{app:DECLeibniz}) the averaging of the vertex positions to the edges, and \(H = \Csf h\) where \(h:\Esf\to\RR\) is defined by \(h_{\eij} = 2\langle n_{\eij}^i, \dsf \fsf_{\eij}^{-1}\rangle = 2\langle n_{\eij}^j, \dsf \fsf_{\ji}^{-1}\rangle\) is the mean curvature of the edge circumsphere. 
  Then the following relationships hold
  \begin{align*}
    \dsf\tau & = \grad\Willmore \\ 
    \dsf\sigma & = \langle \fsf,\grad\Willmore\rangle \\ 
    \dsf\rho & = \fsf\times\grad\Willmore \\
    \dsf\zeta & = \fsf\times(\fsf\times \grad\Willmore) + \fsf\langle \fsf,\grad\Willmore\rangle.
  \end{align*}
  If \(\fsf\) is a critical point of the Willmore energy then \(\tau,\rho,\sigma,\zeta\) are closed.
\end{theorem}
\begin{proof}
  By~\thmref{DECLeibniz}
  \[ \dsf\sigma_i = \langle \fsf_i,\dsf\tau_{i}\rangle  + \Asf^*\langle \dsf f\wedge\tau\rangle_i.\]
  Since \(\langle \tau_{\eij},\dsf \fsf_{\eij}\rangle = 0\) for every edge \(\eij\) the second term is zero. Using the conservation law associated to the translational invariance we obtain the conservation law associated to infinitesimal scalings:
  \[ \dsf\sigma_i = \langle \fsf_i,\dsf\tau_i\rangle = \langle \fsf,\grad\Willmore\rangle_i. \]
  We can similarly use~\thmref{DECLeibniz} to compute
  \begin{equation}\dsf\rho_i = \fsf_i\times \dsf\tau_i + \Asf^*(\dsf \fsf\wedge \tau)_i + \dsf(\tfrac12 H\dsf \fsf)_i. \label{eq:drho}\end{equation}
  From \eqref{WillmoreFlux} we have
  \[
    (\dsf \fsf\wedge\tau)_{\eij} = 
    \dsf \fsf_{\eij}\times\tau_{\eij} 
    = {-2}(n_{il}^i - n_{lj}^j + n_{jk}^j - n_{ki}^i) -  H_{ij}\dsf\fsf_{\eij},
  \]
  and so since the first term vanishes when summing around a vertex, we obtain
  \begin{equation}
    \Asf^*(\dsf \fsf\wedge \tau)_i =  - \dsf(\tfrac12H\dsf\fsf)_{i}
    \label{eq:dftau}
  \end{equation}
  This gives the conservation law associated to infinitesimal rotations
  \[
    \dsf\rho_i = (\fsf\times\grad\Willmore)_i.
  \]
  For the last conservation law, we start from the right hand side and use the two conservation laws derived above:
  \begin{align*}
    & \fsf_i\times (\fsf_i\times\grad\Willmore_i) + \fsf_i\langle \fsf_i,\grad\Willmore_i\rangle = \fsf_i\times\dsf\rho_i + \fsf_i\dsf\sigma_i \\ 
    & = \sum_{\eij}\left[\fsf_i\times\big(\fsf_{\eij}\times\tau_{\eij} + \tfrac{1}{2}H_{ij}\dsf \fsf_{\eij}\big) + \fsf_i\sigma_{\eij}\right] \\
    & = \sum_{\eij}\left[\tfrac12\big(\fsf_i\times(\fsf_j\times\tau_{\eij}) + \fsf_i\times(\fsf_i\times\tau_{\eij})\big) + \tfrac12\fsf_i\times(H_{\eij}\dsf\fsf_{\eij}) + \fsf_i\sigma_{\eij}\right] \\
    & = \sum_{\eij}\left[\tfrac12\big(\fsf_i\times(\fsf_j\times\tau_{\eij}) + \fsf_j\times(\fsf_i\times\tau_{\eij})\big) - \tfrac12\dsf\fsf_{\eij}\times(\fsf_i\times\tau_{\eij}) + \tfrac12\fsf_i\times(H_{\eij}\dsf\fsf_{\eij}) + \fsf_i\sigma_{\eij}\right] \\
    & {= \sum_{\eij} \left[\zeta_{\eij} - \fsf_{\eij}\sigma_{\eij} - \fsf_{\eij}\times(H_{\eij}\dsf\fsf_{\eij}) - \tfrac12\dsf\fsf_{\eij}\times(\fsf_i\times\tau_{\eij}) + \tfrac12\fsf_i\times(H_{\eij}\dsf\fsf_{\eij}) + \fsf_i\sigma_{\eij}\right]} \\
    & = \dsf\zeta_{i} - \sum_{\eij} \left[\fsf_{\eij}\sigma_{\eij} + \fsf_{\eij}\times(H_{\eij}\dsf\fsf_{\eij}) + \tfrac12\dsf\fsf_{\eij}\times(\fsf_i\times\tau_{\eij}) - \tfrac12\fsf_i\times(H_{\eij}\dsf\fsf_{\eij}) - \fsf_i\sigma_{\eij}\right]
  \end{align*}
  Using \(\fsf_{\eij}\sigma_{\eij}-\fsf_{i}\sigma_{\eij} = \tfrac12\dsf\fsf_{\eij}\sigma_{ij}\) and \(\fsf_{\eij}\times(H_{\eij}\dsf\fsf_{\eij}) - \tfrac12\fsf_i\times(H_{\eij}\dsf\fsf_{\eij}) = {\tfrac12 \fsf_j \times (H_{\eij}\dsf\fsf_{\eij}) = \tfrac12 (\fsf_j - \dsf\fsf_{\eij}) \times (H_{\eij}\dsf\fsf_{\eij})} = \tfrac12 \fsf_i \times (H_{\eij}\dsf\fsf_{\eij})\)
  we get that 
  \begin{align*}
    \fsf_i\times (\fsf_i\times\grad\Willmore_i) + \fsf_i\langle \fsf_i,\grad\Willmore_i\rangle & = \dsf\zeta_{i} - \tfrac12\sum_{\eij}\left[\dsf\fsf_{\eij}\sigma_{\eij} + \dsf\fsf_{\eij}\times(\fsf_i\times\tau_{\eij}) + \fsf_j \times (H_{\eij}\dsf\fsf_{\eij}) \right]
  \end{align*}
  Using the Jacobi identity for the cross product, we get \(\dsf\fsf_{\eij}\times(\fsf_i\times\tau_{ij}) = \fsf_i\times(\dsf\fsf_{\eij}\times\tau_{\eij}) - \tau_{\eij}\times(\dsf\fsf_{\eij}\times\fsf_i) = \fsf_i\times(\dsf\fsf_{\eij}\times\tau_{\eij}) - \dsf\fsf_{\eij}\langle\fsf_i,\tau_{\eij}\rangle + \fsf_i\langle\tau_{\eij},\dsf\fsf_{\eij}\rangle\), which we can simplify further to \(\fsf_i\times(\dsf\fsf_{\eij}\times\tau_{\eij}) - \dsf\fsf_{\eij}\sigma_{\eij}\) since \(\langle\tau_{\eij},\dsf\fsf_{\eij}\rangle = 0\) and \(\langle\fsf_i,\tau_{\eij}\rangle = \langle\fsf_{\eij},\tau_{\eij}\rangle\). Hence, using \eqref{dftau} we obtain
  \begin{align*}
    \fsf_i\times (\fsf_i\times\grad\Willmore_i) + \fsf_i\langle \fsf_i,\grad\Willmore_i\rangle = \dsf\zeta_{i} - \tfrac12\sum_{\eij}\left[\fsf_i\times(\dsf\fsf_{\eij}\times\tau_{\eij}) + \fsf_i \times (H_{\eij}\dsf\fsf_{\eij})\right] = \dsf\zeta_{i}.
  \end{align*}
\end{proof}

{ Though we have only explored constraints on the translational flux in this work, it would also be compelling to investigate the influence of the other conservation laws in the context of surface modeling. }

\section{Free Boundary Conditions}
\label{sec:BoundaryConditions}
For the purposes of surface modeling, one is interested in manipulating the geometry of a surface in terms of finitely many parameters. This raises the question about what to do with the boundary. Existing work impose Dirichlet or Neumann boundary conditions, but this is asking too much since it requires the specification of the extrinsic boundary geometry \apriori. Moreover, the free boundary conditions coming from a variational problem do not always control the geometry of the surface well. To extend the domain of practicality of the conformal splines to surfaces with boundary we introduce a natural boundary condition for conformal immersions that only fixes the metric on the boundary up to a scale and has many desirable properties for geometric modeling. {More precisely, we look for conformal immersions \(f:\Sigma\to\RR^3\) that are pointwise conformally equivalent to a reference metric \(g_0\) and with conformal scale factor that is locally constant on the boundary, \ie, \(du|_{\partial\Sigma}\equiv 0\). In the remainder of this section, we will introduce a discretization of these boundary condition based on length (half) cross ratios and derive the discrete Euler-Lagrange equations for the resulting variational problem.}

\begin{figure}[h]
  \centering
  \includegraphics[width=\columnwidth]{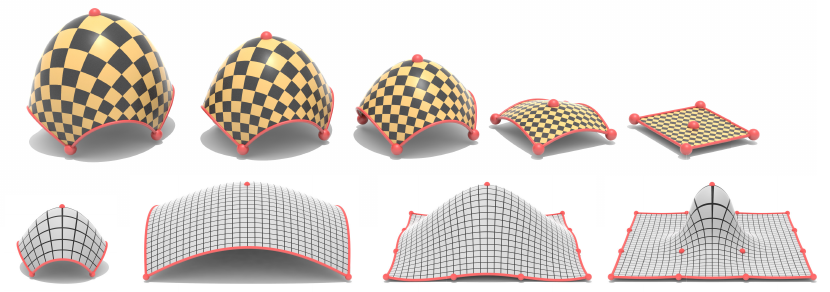}
  \caption{Free boundary conditions enable surface modeling without the specification of any extrinsic geometry (besides the positions of the constrained points).\label{fig:PatchSize}}
\end{figure}

\paragraph*{Boundary First Flattening}
For computing conformal maps into the plane, a natural approach is to consider the Yamabe equation 
\[ \Delta u = K_0 - e^{2u}K, \]
where \(K,K_0\) denote the Gaussian curvature of the metrics \(g = e^{2u}g_0\) and \(g_0\), respectively, since the target curvature 2-form vanishes. The boundary first flattening method of Sawhney and Crane takes this approach~\cite{Sawhney:2017:BFF}, and the algorithm works with either Dirichlet or Neumann boundary conditions for the conformal scale factor. The effect of a Dirichlet boundary condition is straightforward since it determines the metric on the boundary. The Neumann boundary conditions, also known as Cherrier boundary conditions in this context, correspond to prescribing the integrated geodesic curvature of the boundary curve:
\[ \frac{\partial u}{\partial n} = \kappa_0 - e^{u}\kappa. \]
The efficacy of the parameterization algorithm suggests that boundary conditions for the conformal scale factor would also be a natural choice when working with general conformal immersions. 

\paragraph*{Length Half Cross Ratios}
In the context of discrete conformal equivalence, we can relax the Dirichlet boundary conditions on \(u\) to satisfaction modulo constants on each boundary component. ~\\
\setlength{\columnsep}{1em}
\setlength{\intextsep}{0em}
\begin{wrapfigure}{r}{45pt}
  \includegraphics{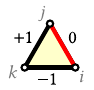}
\end{wrapfigure}
Given a discrete metric \(\ell\in\RR^{|\Esf|}\) we define for every boundary edge \(\eij\in\Esf\setminus\Esf_0\) that is adjacent to the boundary face \(\eijk\) 
the log length half cross ratio \[\log\operatorname{lhcr}_{\eij} = \log\ell_{jk} - \log\ell_{ki}. \]
We now show that the metric on the boundary, up to scale, is encoded in these length half cross ratios.
\setlength{\intextsep}{1em}

\begin{theorem}
  Let \(\ell,\tilde{\ell}\) be discrete metrics on \(\Msf\) that are conformally equivalent. They have the same length half cross ratios if and only if the conformal scale factor is locally constant on the boundary.
  \label{thm:locallyconstant}
\end{theorem}
\begin{proof}
  Since \(\ell,\tilde{\ell}\) are conformally equivalent there exists some \(\usf:\Vsf\to\RR\) so that \(\tilde{\ell}_{\eij} = e^{(\usf_i + \usf_j)/2}\ell_{\eij}\) for every edge \(\eij\in\Esf\). The condition that they have the same length half cross ratio on a boundary edge \(\eij\) is \[ 1 = \frac{\tilde\ell_{\jk}}{\tilde\ell_{\ki}}\frac{{\ell}_{\ki}}{{\ell}_{\jk}} = e^{{(\dsf\usf)_{\eij}}/{2}}, \] which is equivalent to \((\dsf\usf)_{\eij} = 0\). The result follows since the discrete exterior derivative of \(\usf\) is zero on the boundary if and only if \(\usf\) on the boundary is locally constant. 
\end{proof}

\setlength{\columnsep}{1em}
\setlength{\intextsep}{0em}
\begin{wrapfigure}{l}{80pt}
  \includegraphics{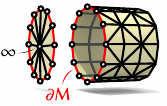}
\end{wrapfigure}
These boundary conditions can also be interpreted as fixing the discrete conformal class of the closed surface obtained by gluing in a disk for each boundary component. The metric on the boundary triangles of the disk has to be chosen so that the boundary lengths agree and so that the remaining two edges have the same lengths. The simplest way to realize this is by attaching a vertex at \(\infty\) to each vertex of a connected component of the boundary. 
This latter interpretation shows that these boundary conditions are related to the free boundary conditions for the M{\"o}bius invariant Willmore energy introduced in~\cite{Bobenko:2005:DWF}: by attaching an imaginary boundary vertex at \(\infty\) to all boundary vertices the circumcircle intersection angles, and therefore also the Willmore energy, can be defined on the boundary---the resulting circumcircle intersection angles on the boundary are equal to the angle between circumcircles and boundary edges. In both cases, these boundary conditions break the existing M{\"o}bius invariance of the setup.
\setlength{\intextsep}{1em}

\paragraph*{Discretization}
Since these boundary conditions are equivalent to fixing the discrete conformal class on a filled in surface, we bundle them into the conformal constraint by constructing a suitable extension \(\Csf\in\RR^{|\Esf|\times|\Esf|}\) of the cross ratio matrix to the boundary. It is defined so that given logarithmic edge lengths \(\lambda:\Esf\to\RR\) the discrete conformal class is given by \(\log\lcr = \Csf\lambda|_{\Esf_0}\) and the length half-cross ratios are given by \(\log\lhcr = \Csf\lambda|_{\Esf\setminus\Esf_0}\). If \(\Msf\) is orientable, it can be built facewise 
\[\Csf = \sum_{\eijk\in\Fsf}\Csf_{\eijk}\]
by summing local \(3\times 3\) matrices 
\[\Csf_{\eijk} = \begin{blockarray}{cccc}
\textcolor{commentblue}{\eij} & \textcolor{commentblue}{\jk} & \textcolor{commentblue}{\ki} \\
\begin{block}{(ccc)c}
  \hphantom{-}0 & \hphantom{-}1 & -1\hphantom{-} & \textcolor{commentblue}{\eij} \\
  -1 & \hphantom{-}0 & \hphantom{-}1\hphantom{-} & \textcolor{commentblue}{\jk} \\
  \hphantom{-}1 & -1 & \hphantom{-}0\hphantom{-} & \textcolor{commentblue}{\ki} \\
\end{block}
\end{blockarray}.\]

\begin{figure}[htpb]
  \centering
  \includegraphics[width=\columnwidth]{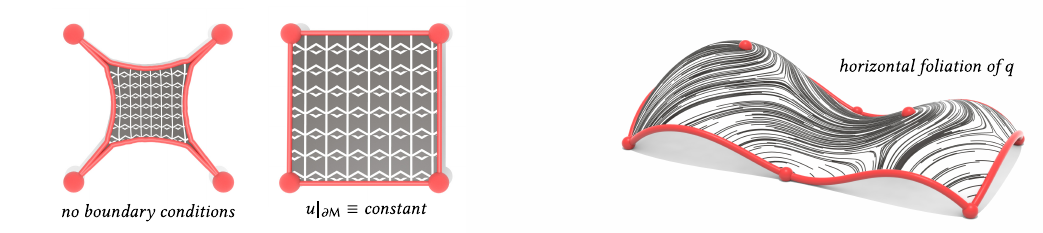}
  \caption{\figloc{Left: }The stabilizing influence of the boundary conditions are well illustrated when considering conformal area minimization with point constraints (here only the four corners of the square are fixed). With no boundary conditions, the surface area of a square patch is able to decrease while preserving the conformal class. The free boundary conditions fixing the metric (up to a scale) prevent this, making the square patch a critical point of area. \figloc{Right: } Visualization of the horizontal foliation of a discrete quadratic differential associated to a conformal surface spline with boundary. The integral curves follow the directions of maximal conformal stress. }
\end{figure}

\paragraph*{Quadratic Differentials}
 Label the boundaries by \(\beta = 1,\dots, b\). Let \(\Vsf_{\beta}\) denote the vertices in the \(\beta\) boundary component. This defines a partition \(\Vsf\setminus \Vsf_0 = \cup_{\beta}\Vsf_\beta\). Define the extended cross-ratio matrix \(\Csf:\RR^{|\Esf|}\to\RR^{|\Esf|}\) by \[ (\Csf\lambda)_{ij} = \lambda_{il} - \lambda_{lj} + \lambda_{jk} - \lambda_{ki}\] if \(ij\in \Esf_0\) and by \[ (\Csf\lambda)_{ij} = \lambda_{jk} - \lambda_{ki} \] if \(ij\) is a boundary edge. Also define \(\Bsf:\RR^{|\Vsf_0|}\times\RR^{b} \to \RR^{|V|}\) by \[\Bsf(u,u^\partial)_i = \begin{cases} u_i & \text{if } i\in \Vsf_0, \\ u^\partial_\beta & \text{if }i\in\Vsf_{\beta}.\end{cases}\] 

 We now use a short exact sequence to describe the Lagrange multipliers for the conformal constraint, including the boundary conditions. 
 \begin{proposition}
  The following sequence is exact \[ 0 \to \RR^{|\Vsf_0|}\times\RR^{b} \xrightarrow{\Asf\circ\Bsf} \RR^{|\Esf|} \xrightarrow{\Csf} \Conf(\mesh) \to 0 \]
 \end{proposition}
\begin{proof}
  By \thmref{locallyconstant} 
\begin{align*}
\lambda\in\ker\Csf & \iff \Csf\lambda|_{E_0} = 0 \text{ and }\Csf\lambda|_{E\setminus E_0} = 0  \iff \exists u\in\RR^V\text{ satisfying } \lambda = \Asf u \text{ and }\Csf\lambda|_{E\setminus E_0} = 0 \\ & \iff  \exists u\in\RR^V\text{ satisfying } \lambda = \Asf u \text{ and } \dsf u|_{E\setminus E_0} = 0  \iff \lambda\in\operatorname{im}(\Asf\circ\Bsf).
\end{align*}
\end{proof}

Therefore, the Euler-Lagrange equations for variational problems constrained to the conformal class and satisfying the free boundary conditions state that \[ \operatorname{grad}\mathcal{E} - \dsf\mu = 0, \] for some \(q \in \ker(\Bsf^*\circ\Asf^*)\), where \(\mu\) is defined in \eqref{extrinsicQD}. A short computation shows
\begin{align*} \ker(\Bsf^*\circ\Asf^*) = \{q:\Esf\to\RR\mid \sum_{j} q_{\eij} = 0 ~\forall i\in\Vsf_0  \text{ and } \sum_{i\in\Vsf_{\beta}}\sum_{\eij}q_{\eij} = 0~\forall \beta=1,\dots,b \}. \end{align*}  

\begin{figure}[htpb]
  \includegraphics[width=\columnwidth]{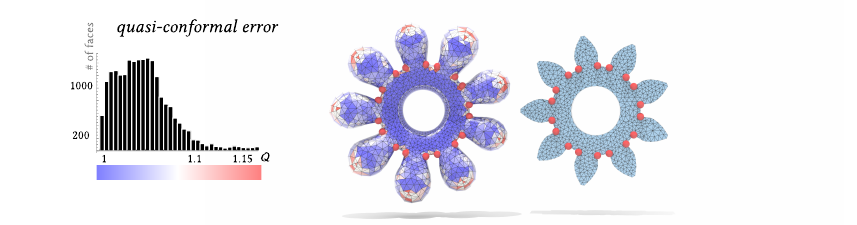}
  \centering
  \caption{Since the notion of discrete conformality is based on length cross ratios the quasi-conformal error is also low as indicated in the histogram (distribution quasi-conformal error \(\mathcal{Q}\) over the faces). Recall that \(\mathcal{Q}=1\) is optimal.}
  \label{fig:gear}
\end{figure}
\section{Numerics and Evaluation}
{To compute the surfaces numerically, we use the algorithm presented in~\cite{Soliman:2021:CWS} for solving conformal variational problems. The knot points are specified by constraints linear in the vertex positions, while the scale factor, conformal class, and boundary conditions are specified by constraints linear in the logarithmic edge lengths. The flux constraint is treated distinctly. Letting \(\nu\in\Omega^2(\Msf^*;\RR^3)\) be defined so that \(\nu_i=0\) for all non-constraint vertices we can formulate the flux constrained algorithm as finding zeros of \((\grad\Willmore - \dsf\mu - \nu, \Csf\log\ell)\) via Newton's method.
Table~\ref{table:Timings} presents the computation time and quasi-conformal error of several examples in the paper. The computation was terminated when the \(L^2\)-norm of the Euler-Lagrange equations was below a tolerance:  
\[\|\dsf\mu - \grad\Willmore\|_{L^2\Omega^2(\Msf^*)} < \varepsilon. \] 
The timings were measured on an Apple M1 Max (2021). We measure the quasi-conformal error on a typical example in \figref{gear}.}

\begin{table}[h]
  \centering
  \caption{Timings of examples in the paper.\label{table:Timings}}
  \begin{tabular}{ c|c|c|c|c|c } 
    & figure & \(|\Vsf|\) & \(|\Esf|\) & iterations & runtime (s) \\ \hline
    ornament & \ref{fig:Icosahedron Scale Constraints} & 30k & 90k & 16 & 14s \\\hline
    flux & \ref{fig:flux} & 35k & 100k & 42 & 55s \\\hline
    patch & \ref{fig:PatchSize} & 10k & 25k & 21 & 9s \\\hline
    giraffe & \ref{fig:Teaser} & 50k & 200k & 105 & 186s \\\hline
    gluing & \ref{fig:Degree 5 Gluing} & 25k & 75k & 18 & 13s \\\hline
    spot & \ref{fig:FixedMetric} & 15k & 35k & 8 & 6s \\\hline
    m{\"o}bius & \ref{fig:MobiusBand} & 5k & 20k & 34 & 17s \\\hline
    4-noid & \ref{fig:BryantSurfaces} & 60k & 200k & 18 & 38s \\\hline
  \end{tabular}
\end{table}

\section{Applications}
{We now consider applications of our boundary conditions and additional conformal parameters to geometric modeling, surface fairing, and mathematical visualization.}

\subsection{Surface Modeling}
\begin{figure}[h]
   \centering
   \includegraphics[width=\columnwidth]{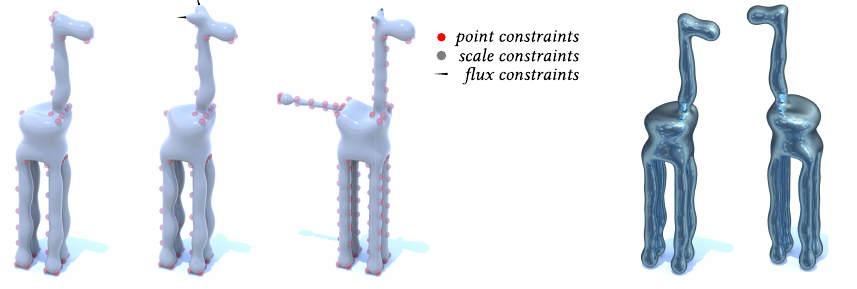}
   \caption{ {Surfaces modeled with conformal splines can be manipulated by specifying point positions, conformal scale factors, and fluxes. \figloc{Left:} A giraffe is modeled by manipulating point constraints on a sphere. \figloc{Middle:} The ears of the giraffe are deformed by changing the tangent planes. \figloc{Right:} Additional point constraints are used to add a tail and prevent the legs from bulging; the ears are instead designed by specifying a positive scale factor relative to the geometry of the right. Since they are Willmore minimizers, they exhibit smooth reflections. \label{fig:Teaser} }}
\end{figure}

Conformal surface splines provide a low dimensional subspace of the space of smooth surfaces interpolating a collection of point constraints. \figref{Teaser} was designed by manipulating point constraints, conformal knot points, and scale constraints. In the limit, the surfaces are smooth away from the point constraints, and they are at worst \(C^1\) at the constraint points. One of the main advantages of working with fully nonlinear splines pertains to how they are glued together. Splines that are based on tensor product representations can only be glued together in such a way that four patches meet at a point, whereas no such restriction exists when the surfaces are parameterized over abstract Riemann surfaces. One can simply glue together any two domains with consistent boundary conditions to get a discrete Riemann surface over which the new smooth surface seamlessly interpolates the constraint points (\figref{Degree 5 Gluing}).

\begin{figure}[!h]
  \includegraphics[width=\columnwidth]{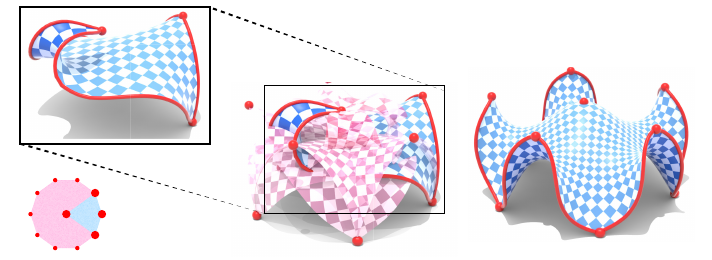}
  \centering
  \caption{{Gluing spline patches together at irregular degree vertices pose no problems for our method, which is formulated over abstract two-dimensional domains. We start with the blue patch (\figloc{left}), and minimize the Willmore energy in a fixed discrete conformal class with free boundary conditions. Duplicating and rotating the patch does not produce a smooth surface (\figloc{middle}), but this can be smoothed away after gluing the surface intrinsically (\figloc{right}).}}
  \label{fig:Degree 5 Gluing}
\end{figure}

\subsection{Surface Fairing}
Surface fairing is a central tool in the geometry processing toolkit, and is an application that is well handled by conformal scale constraints. 
Fixing the conformal scale factor to be zero in selected regions thereby enforces that there the deformations remain isometric. This formulation is better suited when the preservation of existing geometric features is desired (\figref{FixedMetric}).
This approach is complemented by existing conformal fairing algorithms: \cite{Crane:2013:RFC} evolves the conformal immersion by the gradient flow of a curvature energy for short time, and \cite{Soliman:2021:CWS} minimizes a weighted combination of the Willmore energy and a fidelity term.

\begin{figure}[htpb]
  \centering
  \includegraphics[width=\columnwidth]{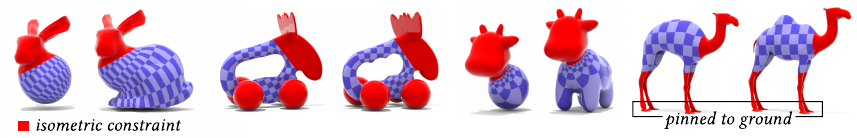}
  \caption{Preserving geometric features during fairing can be achieved by using vanishing conformal scale factor constraints and minimizing the Willmore energy to optimality. In the case of the camel, we also use point constraints to fix the bottom of the feet to the ground. 
  \label{fig:FixedMetric}}
\end{figure}

\subsection{Mathematical Visualization}
Constrained Willmore surfaces are differential geometric objects that can be illustrated and investigated through the numerical computation of their discrete counterparts. 
\paragraph*{Conformal M{\"o}bius Bands}
A natural differential geometric question is to determine the optimal geometric realization of a M{\"o}bius band in \(\RR^3\). 
The Willmore energy alone does not also provide a meaningful way to answer this question due to the following construction (though an infinite family of Willmore critical M{\"o}bius bands was recently constructed in~\cite{Dorfmeister:2024:EPH}). Boy's surface is a Willmore real projective plane. By taking a thin strip out of the surface we obtain a Willmore M{\"o}bius band, and by decreasing the width of this strip the area of the M{\"o}bius band converges to zero and consequentially the Willmore energy also converges to zero. Fixing the conformal class of the surface prevents this kind of degeneration and provides an optimal geometric realization in each conformal class (\figref{MobiusBand}). 
\begin{figure}[htpb]
  \centering
  \includegraphics[width=\columnwidth]{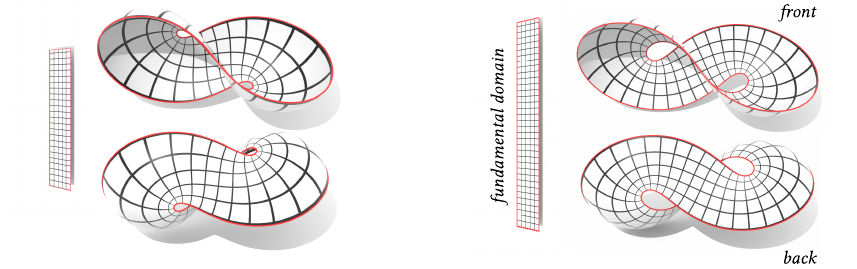}\caption{
    We compute Willmore minimizing M{\"o}bius bands in two conformal classes, described by identifying opposite sides of the flat rectangular strips (fundamental domain) illustrated above. Two views are visualized of the minimizing realizations of these conformal M{\"o}bius band in \(\RR^3\).
    \label{fig:MobiusBand}}
\end{figure}

\paragraph*{Bryant Surfaces}
Bryant surfaces are constant mean curvature surfaces in hyperbolic 3-space and provide an important family of examples of constrained Willmore surfaces in \(\RR^3\)~\cite{Bryant:1987:SMC,Bohle:2009:BSS}. 
\figref{BryantSurfaces} illustrates how by using the conformal scale factor constraints one can obtain discrete surfaces that closely resemble these smooth surfaces, and suggest that the Bryant surfaces may also be classified as constrained Willmore minimizers subject to point and scale constraints.

\begin{figure}[htpb]
  \centering
  \includegraphics[width=\columnwidth]{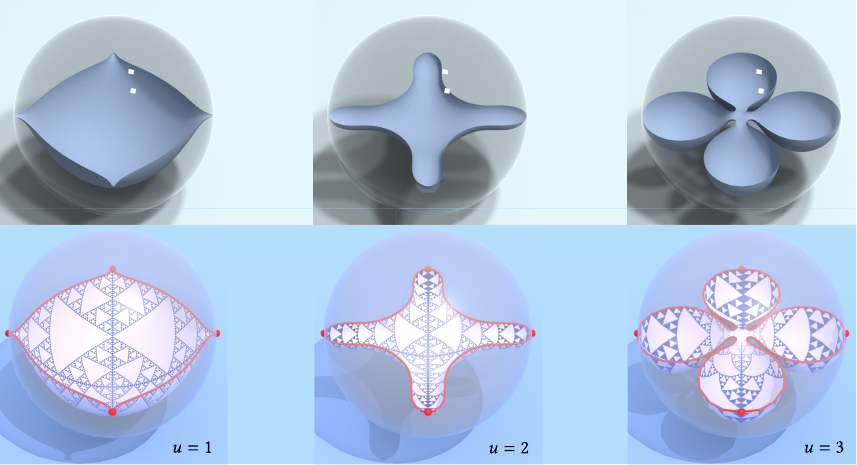}
  \caption{{\figloc{Top row:} visualizations of constant mean curvature 4-noids in the Poincar{\'e} ball model of hyperbolic 3-space from~\cite{Schmitt:GG}. \figloc{Bottom row:} discrete constrained Willmore surfaces with four point and four scale constraints (at the marked points) in \(\RR^3\). The scale factors are given relative to the round metric on \(S^2\). \label{fig:BryantSurfaces}}}
\end{figure}

\section{Discussion}
{We utilize constrained Willmore surfaces and their conformal properties (knot points, scale and flux constraints) for free-form surface modeling since they are specified by finitely many parameters. We also introduced free boundary conditions for conformal surface splines that do not require the specification of any extrinsic geometry. There are two primary limitations of the proposed framework that we aim to address in future work. The first is that drastic changes in the positions and conformal scale factor of constraint vertices typically lead to regions of the surface with high area distortion that will be poorly approximated at the original mesh resolution. This can be somewhat mitigated by the use of higher order (\eg, spherical) elements~\cite{Kilian:2023:MSF}, but we are also interested in exploring the possibility of M{\"o}bius geometric subdivision schemes related to the notion of constrained Willmore surfaces. The second limitation is that the iterations of the optimization require the solution of large linear systems making the interactive manipulation of the conformal splines still out of reach of existing methods. There is a relationship between Willmore surfaces and minimal surfaces in hyperbolic space~\cite{Richter:1997:CMR} that may provide a more efficient numerical method for computing conformal surface splines.}

\section*{Acknowledgements}
\noindent
This work was supported in part by the IST initiative at Caltech, the Einstein Foundation in Berlin, the Hausdorff Center at the University of Bonn, the Deutsche Forschungsgemeinschaft (DFG - German Research Foundation) - Project-ID 195170736 - TRR109 ``Discretization in Geometry and Dynamics,'' and Side Effects Software, Inc.

\bibliographystyle{elsarticle-num} 
\bibliography{ConformalSurfaceSplines}     

\appendix

\section{DEC Leibniz Rule}
\label{app:DECLeibniz}
The conservation laws of the discrete Willmore can be derived in a straightforward manner once we define some new spaces of discrete differential forms. Define the spaces of discrete differential 0-forms and 2-forms supported on the edges
\[
  \Omega^0_{\Esf}(\Msf)\coloneqq\{\phi:\Esf\to\RR\},\qquad \Omega^2_{\Esf}(\Msf)\coloneqq\{\sigma:\Esf\to\RR\},
\]
respectively. Although elements of both are identified with unoriented values per edge they are distinct spaces (with \(\Omega^2_{\Esf}\) being identified with the dual space of \(\Omega^0_{\Esf}\)). 
The averaging operator (\eqref{averaging_operator}) can now be interpreted as a map \(\Asf:\Omega^0(\Msf)\to\Omega^0_{\Esf}(\Msf)\) and its adjoint \(\Asf:\Omega^2_{\Esf}(\Msf)\to\Omega^2(\Msf^*)\). There is the natural elementwise multiplication \(\Omega^0_{\Esf}(\Msf)\times\Omega^1(\Msf^*)\to\Omega^1(\Msf^*)\): for \(\phi\in\Omega^0_{\Esf}(\Msf)\) and \(\alpha\in\Omega^1(\Msf^*)\), \((\phi\alpha)_{\eij} = f_{\eij}\alpha_{\eij}\). We will also use the discrete wedge product \(\wedge:\Omega^1(\Msf)\times\Omega^1(\Msf^*)\to\Omega^2_{\Esf}(\Msf)\) between primal and dual 1-forms: for \(\alpha\in\Omega^1(\Msf)\) and \(\beta\in\Omega^1(\Msf^*)\), \((\alpha\wedge\beta)_{\eij} = \alpha_{\eij}\beta_{\eij}\). These spaces of differential forms fit together so that there is a product rule between primal 0-forms and dual 1-forms (or dually, dual 0-forms and primal 1-forms).

\begin{theorem}
  \label{thm:DECLeibniz}
  Let \(\phi\in\Omega^0(\Msf)\) and \(\alpha\in\Omega^1(\Msf^*)\) then \[ \dsf(\phi\alpha) = \phi\,\dsf\alpha + \Asf^*(\dsf \phi\wedge \alpha).\] 
\end{theorem}
\begin{proof}
  For a vertex \(i\in\Vsf\) we have \begin{align*}
    \dsf(\phi\alpha)_{i} & = \sum_{\eij}\frac{\phi_i + \phi_j}{2}\alpha_{ij} = \sum_{\eij}(\phi_i + \frac12\dsf\phi_{\eij})\alpha_{ij} \\
    & = \phi_i\sum_{\eij}\alpha_{ij} + \frac12\sum_{\eij}\dsf\phi_{\eij}\alpha_{\eij} = \phi_i (\dsf\alpha)_i + (\Asf^*(\dsf \phi\wedge\alpha))_i 
  \end{align*}
\end{proof}

\end{document}